\documentclass[12pt]{amsart}

\usepackage{xcolor}

\usepackage{hyphenat}
\usepackage{amssymb}

\usepackage{enumitem}

\newcommand{\mb}{\mathbb}
\newcommand{\mc}{\mathcal}
\newcommand{\K}{{\mc K}}
\newcommand{\eul}{\mathfrak}
\newcommand{\A}{\eul A}
\newcommand{\Ao}{{\eul A}_{\scriptscriptstyle 0}}
\newcommand{\M}{\mc M}
\newcommand{\D}{\mc D}
\newcommand{\F}{\mc F}
\newcommand{\QA}{{\mathcal Q}_{\Ao}(\A)}
\newcommand{\IA}{{\mathcal I}_{\Ao}(\A)}
\newcommand{\PA}[1]{{\mathcal P}_{\Ao}^{#1}(\A)}
\def\H{\mc H}
\newcommand{\idop}{{\mb I}}
\newcommand{\id}{{\sf e }}
\newcommand{\vp}{\varphi}
\newcommand{\wmultt}{\mult}
\newcommand{\mult}{\mbox{\raisebox{1pt}{$\scriptscriptstyle{	\square}$}}} 
\newcommand{\wmult}{\mbox{\raisebox{1pt}{$\scriptscriptstyle{
				\diamond}$}}}
			\newcommand{\Lc}{{\mc L}}
			\newcommand{\ad}{^{\mbox{\scriptsize $\dag$}}}
\newcommand{\LDH}{{\mathcal L}\ad(\D,\H)}
\newcommand{\ze}{_{\scriptscriptstyle 0}}
\def\x{\relax\ifmmode {\mbox{*}}\else*\fi}
\newcommand{\LD}{{\mc L}\ad(\D)}
\newcommand{\up}{\raisebox{0.7mm}{$\upharpoonright$}}
\newcommand{\ip}[2]{\langle{#1}|{#2}\rangle}
\newcommand{\LDb}{{\mc L}\ad(\D)_b}
\newcommand{\sft}{{\sf t}\!{\sf t}}
\newcommand{\B}{{\eul B}}
\newcommand{\mB}{\mc B}
\newcommand{\ltaum}{\tau}
\newcommand{\FF}{\mathfrak F}
\newcommand{\LpDr}{{\mathcal L}\ad(\D_{\scriptscriptstyle\pi})}
\newcommand{\KII}{\K_{\scriptscriptstyle{{\mc I}}}}
\newcommand{\KI}{\K_{\scriptscriptstyle{\M}}}
\newcommand{\bdd}{\A_{\scriptscriptstyle{\rm b}}(\KI)}
\newcommand{\bdds}{\A_{\scriptscriptstyle{\rm b}}^{\scriptscriptstyle{\M}}}
\newcommand{\LDHr}{{\mathcal L}\ad(\D_\pi,\H_\pi)}
\newcommand{\leqm}{\leq_{\scriptscriptstyle{\M}}}

\newcommand{\taum}{\tau^{\scriptscriptstyle{\M}}}
\newcommand{\rnb}{^{\scriptscriptstyle{\M}}_{\scriptscriptstyle{\rm b}}}

\newcommand{\mrep}{{\sf Rep}^{r, \M}(\A,\Ao)}
\newcommand{\irep}{{\sf Rep}^{r}(\A,\Ao)}

\newcommand{\BH}{\mc{B}(\mathcal{H})}

\makeatletter
\@namedef{subjclassname@2020}{%
  \textup{2020} Mathematics Subject Classification}
\makeatother

\usepackage[T1]{fontenc}

\newtheorem{theorem}{Theorem}[section]

\newtheorem{lemma}[theorem]{Lemma}
\newtheorem{proposition}[theorem]{Proposition}

\theoremstyle{definition}
\newtheorem{definition}[theorem]{Definition}
\newtheorem{remark}[theorem]{Remark}
\newtheorem{example}[theorem]{Example}

\numberwithin{equation}{section}

\begin{document}


\baselineskip=17pt


\title[Topological aspects of quasi *-algebras]{Topological aspects of quasi *-algebras with sufficiently many *-representations}

\author[G. Bellomonte]{Giorgia Bellomonte}

\email{giorgia.bellomonte@unipa.it}

\author[C. Trapani]{Camillo Trapani}

\email{camillo.trapani@unipa.it}

\address{Dipartimento di Matematica e Informatica, Universit\`a degli Studi  di Palermo, Via Archirafi n. 34,  I-90123 Palermo, Italy}
\date{}

\begin{abstract}Quasi *-algebras possessing a sufficient family $\M$ of invariant positive sesquilinear forms carry several topologies related to $\M$ which make every *-representation continuous. This leads to define the  class  of locally convex quasi GA*-algebras whose main feature consists in the fact that the family of their bounded elements, with respect to the family  $\M$, is a dense C*-algebra.
\end{abstract}



\maketitle

\section{Introduction} 

{Locally convex quasi *-algebras $(\A[\tau],\Ao)$ arise often when taking
	{the completion $\A:= \widetilde{\A_0}[\tau]$ of a locally convex *-algebra
		$\A_0[\tau]$}
	with separately (but not jointly) continuous multiplication (this was, in fact, the case considered at an early stage of the theory, concerning applications in quantum physics).		
	Concrete examples are provided by families of operators acting in rigged Hilbert spaces or by certain families of unbounded operators acting on a common domain $\D$ of a Hilbert space $\H$.
	For a synthesis of the theory and of its applications we refer to \cite{FT_book}.
	
	The study of this structure  and the analysis performed also in 	\cite{ ant_be_ct, bell_db_ct,  bell_ct, Frag2, ctbound1} made it clear that the most regular situation occurs when the locally convex quasi *-algebra $(\A[\tau],\Ao)$ under consideration possesses a {\em sufficiently rich} family $\IA$ of invariant positive sesquilinear forms on $\A\times \A$ (see below for definitions); they allow  a GNS construction similar to that defined by  a positive linear functional on a *-algebra $\Ao$.
	The basic idea where this paper moves from, is to consider a quasi *-algebra $(\A,\Ao)$ where one can introduce a locally convex topology by means of the set of sesquilinear forms $\IA$ itself. In the best circumstances we expect a behavior analogous to that of a *-algebra ${\mathfrak B}_0$ whose topology can be defined via families of C*-seminorms:		
	$$p_M(x)= \sup_{\omega\in M; \omega(\id)=1} \omega(x^*x)^{1/2}$$
	where $M$ is a convenient set of positive linear functionals on ${\mathfrak B}_0$  \cite{yood}. 
	
	For this reason, we start from a pure algebraic setup, i.e., $(\A,\Ao)$ is a quasi *-algebra and we suppose that  it has a sufficiently large $\IA$ (in the sense that, for some convenient subset $\M\subset\IA$ and for every $a\in \A, a\neq0$ there exists $\vp\in \M$ such that $ \vp(a,a)>0$). 
	Starting from this set $\M$, we undertake the construction of locally convex topologies on $\A$ selecting in particular those under which each (sufficiently regular) *-representation is continuous. This analysis leads to the selection of a class of locally convex quasi *-algebras $(\A[\tau], \Ao)$ (called locally convex quasi GA*-algebras) whose bounded elements constitute a C*-algebra. 	The paper is organized as follows.
	In Section \ref{Sect_Prelim} some preliminary notions on quasi *-algebras, their topologies and their representations are summarized.
	In Section \ref{Sect_Order} we introduce the order defined by a family $\M\subset \IA$ whose related wedge becomes a cone when the family $\M$ is sufficiently rich.
	In Section \ref{Sect_Bdd}, given $\M\subset \IA$, we introduce two notions of bounded elements: those bounded with respect to a family $\M$ and those related to the  order defined by $\M$. These two notions turn out to be equivalent  and every *-representation produces a bounded operator when acting on a bounded element.
	In Section \ref{Sect_Topologies} the topologies generated by a family $\M\subset\IA$ are investigated and in Section \ref{Sect_GA} we finally introduce locally convex quasi GA*-algebras and study some properties of them. Locally convex quasi GA*-algebras are characterized by the fact that their topology is equivalent to that generated by some $\M\subset \IA$ as in Section \ref{Sect_Topologies}.

\section{Basic definitions and facts} \label{Sect_Prelim}

We begin with some preliminaries; we refer to \cite{FT_book} for details.

\medskip  
A {\em quasi *-algebra} $(\A, \Ao)$ is a pair consisting of a vector space $\A$ and a *-algebra $\Ao$ contained in $\A$ as a subspace and such that
\begin{itemize}
	\item[(i)] $\A$ carries an involution $a\mapsto a^*$ extending the involution of $\Ao$;
	\item[(ii)] $\A$ is  a bimodule over $\A_0$ and the module multiplications extend the multiplication of $\Ao$. In particular, the following associative laws hold:
	\begin{equation}\notag \label{eq_associativity}
		(xa)y = x(ay); \ \ a(xy)= (ax)y, \quad \forall \ a \in \A, \  x,y \in \Ao;
	\end{equation}
	\item[(iii)] $(ax)^*=x^*a^*$, for every $a \in \A$ and $x \in \Ao$.
\end{itemize}

The
\emph{identity} of $(\A, \Ao)$, if any, is a necessarily unique element $\id\in \Ao$, such that
$a\id=a=\id a$, for all $a \in \A$.

We will always suppose that
\begin{align*}
	&ax=0, \; \forall x\in \Ao \Rightarrow a=0 \\
	&ax=0, \; \forall a\in \A \Rightarrow x=0. 
\end{align*}
These two conditions are clearly satisfied if $(\A, \Ao)$ has an identity $\id$.
\begin{definition}
A  quasi *-algebra $(\A, \Ao)$ is said to be  {\em locally convex} if $\A$ is a locally convex vector space, with a topology $\tau$ enjoying the following properties
\begin{itemize}
	\item[{\sf (lc1)}] $x\mapsto x^*$, \ $x\in\A\ze$,  is continuous;
	\item[{\sf (lc2)}] for every $a \in \A$, the maps  $x \mapsto ax$ and
	$x \mapsto xa$, from $\A\ze$ into $\A$, $x\in \A\ze$, are continuous;
	\item[{\sf (lc3)}] $\overline{\Ao}^\tau = \A$; i.e., $\Ao$ is dense in $\A[\tau]$.
\end{itemize}

In particular, if $\tau$ is a norm topology, with norm $\|\cdot\|$, and
\begin{itemize}
	\item[{\sf (bq*)}] $\|a^*\|=\|a\|, \; \forall a \in \A$
\end{itemize}
then, $(\A[\|\cdot\|], \Ao)$ is called a {\em normed  quasi *-algebra} and a {\em Banach  quasi *-algebra} if the normed vector space $\A[\|\cdot\|]$ is complete.
\end{definition}

\medskip
Let $\D$ be a dense vector subspace of a Hilbert space $\H$. Let us consider the following families of linear operators acting on $\D$:
\begin{align*}		{\LDH}&=\{X \mbox{ closable}, D(X)=\D;\; D(X\x)\supset \D\} \\	{\Lc^\dagger(\D)}&=\{X\in \LDH: X\D\subset \D; \; X\x\D\subset \D\} \\	{\Lc^\dagger(\D)_b} &=\{Y\in \Lc^\dagger(\D); \, \overline{Y} \mbox{ bounded} \},\end{align*}
where $\overline{Y}$ denotes the closure of $Y$.
The involution in $\LDH$ is defined by		$X^\dag := X\x\upharpoonright \D$, the restriction of $X\x$, the adjoint of $X$, to $\D$.

The set $\LD$ is a *-algebra; more precisely, it is the maximal O*-algebra on $\D$.
{(for the theories of O*-algebras and *-representations we refer to \cite{Shmud Unboun})}.

\smallskip 		{	Furthermore, $\LDH$ is also a {\em partial *-algebra} \cite{Ant1}  with respect to the following operations: the usual sum $X_1 + X_2 $,
	the scalar multiplication $\lambda X$, the involution $ X \mapsto X\ad := X\x \up {\D}$ and the \emph{(weak)}
	partial multiplication 
	\begin{equation} \label{eqn_wmult}
		X_1 \mult X_2 = {X_1}\ad\x X_2,
	\end{equation} defined whenever $X_2$ is a weak right multiplier of
	$X_1$ (we shall write $X_2 \in R^{\rm w}(X_1)$ or $X_1 \in L^{\rm w}(X_2)$), that is, whenever $ X_2 {\D} \subset
	{\D}({X_1}\ad\x)$ and  $ X_1\x {\D} \subset {\D}(X_2\x).$

	\medskip}

The following topologies on $\LDH$ will be used in this paper.

The \emph{weak topology} ${\sf t}_w$ on $\LDH$ is  defined by the seminorms
$$ r_{\xi, \eta}(X)=|\ip{X\xi}{\eta}|, \quad X \in \LDH,\, \xi, \eta \in \D.$$

The \emph{strong topology} ${\sf t}_s$ on $\LDH$ is  defined by the seminorms
$$ p_\xi(X)=\|X\xi\|, \quad X \in \LDH, \, \xi \in \D.$$

\noindent The \emph{strong* topology} ${\sf t}_{s^\ast}$ on  $\LDH$ is usually defined by the seminorms
$$p^*_\xi (X)= \max\{\|X\xi\|, \|X\ad\xi\|\}, \,  \xi \in \D.$$

Then,		$(  {\LDH}[{\sf t}_{s^\ast}],    {\Lc^\dagger(\D)_b})$
is a complete locally convex quasi *-algebra \cite[Section 6.1]{FT_book}.

\medskip
{Let us denote by $t_\dag$ 
	the graph topology  on $\D$ defined by the set of seminorms
	$$\xi \in \D \to \|X\xi\|; \; X\in \LDH.$$
	The  family of all bounded subsets of $\D[t_\dag]$  is denoted by $\B$. }

	{	  We will indicate by $\sft_u$ , $\sft^u$ and $\sft^u_*$, respectively, the {\em uniform} topologies defined by the following families of seminorms:}
	
	\noindent for  $\sft_u$ : $p_\mB(X) = \sup_{\xi, \eta \in \mB}|\ip{X\xi}{\eta}|, \quad \mB \in \B$;
	
	\noindent for  $\sft^u$ : $p^\mB(X) = \sup_{\xi\in \mB}\|X\xi\|, \quad \mB \in \B$;
	
	\noindent for  $\sft^u_*$ : $p^\mB_*(X) = \max\{p^\mB(X), p^\mB(X\ad)  \}, \quad \mB \in \B$.
	
	{It is easy to see that $\sft_u \preceq\sft^u \preceq\sft^u_*$:}
	$$ p_\mB(X) \leq \gamma_\mB \, p^\mB(X)\leq \gamma_\mB p^\mB_*(X)
	, \quad \forall X \in \LDH;$$
	moreover, $$p_\mB( X\ad \mult X)=p^\mB (X)^2 \mbox{  whenever } X\ad \mult X \mbox{ is well--defined}.$$
	
	As shown in \cite[Proposition 4.2.3]{Ant1}
	$\LDH[\sft^u_*]$ is complete.

\begin{definition} \label{defn_starrep}Let $(\A,\Ao)$ be a quasi *--algebra and $\D_\pi$ \index{$\D_\pi$} a dense domain
in a certain Hilbert
space $\H_\pi$. \index{$\H_\pi$}  A linear map $\pi$ from $\A$ into $\LDHr$ is called  a
*--\emph{representation} of \index{*--representation} $(\A, \Ao)$,
if the following properties are fulfilled:
\begin{itemize}
	\item[(i)]  $\pi(a^*)=\pi(a)^\dagger, \quad \forall \ a\in \A$;
	\item[(ii)] for $a\in \A$ and $x\in \Ao$, $\pi(a)\mult\pi(x)$ is well--defined and {$\pi(a)\mult \pi(x)=\pi(ax)$}.
\end{itemize}

If $(\A,\Ao)$ has a unit $\id
\in \Ao$, we assume that for every *-representation $\pi$ of $(\A,\A_0)$,  $\pi(\id)={\idop_{\D_\pi}}$, the
identity operator on  the space $\D_\pi$.
\end{definition}

If {$\pi_o:= \pi\upharpoonright{\Ao}$} is a *--representation of the *--algebra $\Ao$ into $\LpDr$ we say that
$\pi$ is a {\em qu*--representation} of $(\A,\A_0)$.\index{qu*--representation}

A *--representation $\pi$ is called {\em bounded}
if $\pi(a)$ is a bounded operator in $\D_\pi$, for every $a \in \A$.

\medskip

Let $(\A,\Ao)$ be a quasi *--algebra.  We
denote by $\QA$ \index{$\QA$} the set of all sesquilinear forms on $\A \times \A$, such that
\begin{itemize}
	\item[(i)]$\vp$ is positive, i.e.,  $\vp(a,a)\geq 0, \quad \forall \ a \in \A$;
	\item[(ii)]$\vp(ax,y)=\vp(x, a^*y), \quad \forall \ a \in \A, \ x,y \in \Ao$.
\end{itemize}

{For every $\vp \in \QA$, the set	
			$$ N_{\varphi} := \big\{ a \in \A :
			\varphi(a,a) = 0 \big\} = \big\{ a \in \A : \varphi(a,b) = 0, \ \forall \
			b \in \A \big\}.
			$$
is a subspace of $\A$.}

Let $\lambda_{\varphi} : \A \rightarrow \A/N_{\varphi}$
be the usual quotient map and for each $ a
\in \A$, let $\lambda_{\varphi}(a)$ \index{$\lambda_{\varphi}(a)$} be the corresponding coset of $
\A/N_{\varphi}$, which contains $a$. An inner product $\ip{\cdot}{\cdot}$ is then defined
on $\lambda_\varphi(\A) =\A/N_{\varphi}$
by
\begin{align*} \label{lavar}
\ip{\lambda_{\varphi}(a)} {\lambda_{\varphi}(b) } := \varphi(a,b), \quad \forall \; a,b \in \A.
\end{align*}
Denote by $\H _{\vp}$ the Hilbert space obtained by the
completion of the pre--Hilbert space $\lambda_\varphi(\A).$

{ \begin{definition} We denote by $\IA$ the subset of forms $\vp \in \QA$ for which $\lambda_\vp(\Ao)$ is dense in $\H_\vp$. {Elements of $\IA$ are also called {\em invariant positive sesquilinear forms} or briefly {\em ips-forms}}.

Moreover, if $(\A[\tau], \Ao)$ is a locally convex quasi *-algebra, we denote by $\PA{\tau}$   the family of elements $\vp$ of $\QA$ that are jointly $\tau$-continuous; i.e., there exists a continuous seminorm $p_\sigma$ such that $$|\vp(a,b)|\leq p_\sigma(a)p_\sigma(b), \quad \forall a,b \in \A .$$		
\end{definition}}

The sesquilinear forms of $\IA$ allow  building up  a GNS-represen\-ta\-tion \cite{FT_book}. Indeed,
\begin{proposition} \label{idix} Let $(\A,\Ao)$ be a quasi *--algebra with unit
$\id$ and $\vp$ a sesquilinear form on $\A \times \A$. The following
statements are equivalent:
\vspace{-1mm}
\begin{itemize}
\item[{\em (i)}] $\vp \in \IA$.
\item[{\em (ii)}] There exist a Hilbert space $\H _\vp$,\index{$\H_\vp$} a dense domain $\D_\vp$ \index{$\D_\vp$} of the Hilbert space $\H _\vp$
and a closed cyclic *--representation
$\pi_\vp$ in ${\mathcal L}\ad(\D_\vp,\H _\vp)$, with cyclic
vector $\xi_\vp$ {\em (in the sense that $\pi_\vp(\Ao)\xi_\vp$ is dense in $\H _\vp$)},
such that \begin{equation*} \vp(a,b)
	=\ip{\pi_\vp(a)\xi_\vp}{\pi_\vp(b)\xi_\vp}, \quad \forall \ a, b \in \A.\end{equation*}
\end{itemize}
\end{proposition}

\begin{remark} {  The *-representation $\pi_\vp$ is in fact obtained by taking the closure of the *-representation $\pi^\circ_\vp$ defined on $\lambda_\vp(\Ao)$ by $$\pi^\circ_\vp(a)\lambda_\vp(x)= \lambda_\vp(ax)
\quad a\in \A, x\in \Ao. $$ }
\end{remark}

{	\medskip 	If $\M\subset \IA$ is {\em rich enough} (in the sense that, for every $a\in \A$, there exists $\vp\in \M$, such that $\vp(a,a)>0$) then we can introduce a partial multiplication as in \cite[Definition 3.1.30]{FT_book}.

	Indeed, in this case, 
	we say that the {\it weak} multiplication, $a\wmult b, \ a, b \in \A$, is well--defined if there exists $c\in\A$, such that
	\begin{equation}\label{def-3.1.30}
		\varphi(bx,a^*y)=\varphi(cx,y), \quad \forall \ x,y\in\Ao \ \mbox{ and } \ \varphi\in\M.
	\end{equation}
	In this case, we put $a\wmult b:=c$. }

With these definitions, we conclude that (\cite[Proposition 4.4]{att_2010})
$\A$  is also a  partial *--algebra with respect to the weak multiplication $\wmult$.

{	\begin{remark} The uniqueness of $c= a\wmult b$ is guaranteed by Proposition \ref{prop_2.4} below. Clearly this multiplication depends on the family $\M$. 
	\end{remark}}

\section{Families of forms and order structure} \label{Sect_Order}
As discussed extensively in \cite{FT_book}, the notion of bounded element of a locally convex quasi *-algebra reveals to be important for undertaking a spectral analysis in this structure. We propose here two different approaches similar to those developed in \cite{att_2010} but without the continuity assumptions made therein.

\smallskip
Before going forth, we introduce some notions needed in what follows. In particular, 
in analogy to \cite{yood}, 
\begin{definition} A subset $\M \subset \IA$ is said to be

\begin{description}
	\item[ {\em balanced} ] if  $\vp\in\M$ implies $\vp^x \in \M$, for every $x \in \Ao$, where $\vp^x(a,b):=\vp(ax,bx)$ for all $a,b\in \A$.
	\item[{\em sufficient}] if it is balanced and if, for every $a\in \A\setminus\{0\}$, there exists $\vp\in \M$ such that $\vp(a,a)>0$.
\end{description}
\end{definition}

{ \begin{remark} \label{lemma: vpx}If $(\A[\tau], \Ao)$ is a locally convex quasi *-algebra, then
	\begin{itemize}\item[(a)] $\PA{\tau} \subset \IA$;
		\item[(b)] $\PA{\tau}$ is balanced.
	\end{itemize}
	\end{remark}	}
\medskip

The following proposition allows us to deal with notion of sufficiency in other equivalent ways.
\begin{proposition}\label{prop_2.4} Let $(\A,\A_0)$ be a quasi *-algebra, $\M$ a subset of $\IA$  and  $a\in\A$. Then the following are equivalent:
	\begin{itemize}
		\item[i)] $\vp(ax,x)=0$, for all $\vp\in\M$, $x\in\A_0$;
		\item[ii)] $\vp(ax,y)=0$, for all $\vp\in\M$, $x,y\in\A_0$;
		\item[iii)]$\vp(ax,ax)=0$, for all $\vp\in\M$, $x\in\A_0$.\end{itemize} 
	If $(\A,\Ao)$ has unit $\id$ and $\M$ is balanced, then the previous statements are equivalent to
	\begin{itemize}
		\item[iv)]$\vp(a,a)=0$, for every $\vp\in\M$.
	\end{itemize}
	\end{proposition}

In the case of a locally convex quasi *-algebra $(\A[\tau], \Ao)$,  positive elements have been defined in  \cite{Frag2} as the members of the closure $\A^+:=\overline{\Ao^+}^{\,\tau}$,  where
\begin{equation*} {\Ao^+}:=\left\{\sum_{k=1}^n x_k^* x_k, \, x_k \in \Ao,\, n \in {\mb N}\right\}.\end{equation*}
Here, as we have anticipated, we will start from a quasi *-algebra without a topology and  we will introduce the notion of positive element via a family $\M$ of forms of $\IA$.

\begin{definition} Let $(\A,\A_0)$ be a quasi *-algebra. We call	
$\M$-{positive} an element $a\in \A$ such that
$$ \vp(ax,x)\geq 0, \quad \forall \vp \in \M, \forall x \in \Ao.$$ 
We put
\begin{align*}\KI &:=\{a\in \A:\,\vp(ax,x)\geq 0, \,\,    \forall\vp\in\M, \forall x\in\A_0\}.\end{align*}
If $\M=\IA$ we denote the corresponding set by $\KII$.
\end{definition}

{ \begin{lemma} Let $(\A,\Ao)$ be a quasi *-algebra with a sufficient $\M\subset\IA$.  If $a$ is $\M$-positive, then $a=a^*$.
	\end{lemma}
	
	\begin{proof} The conclusion is a consequence of Proposition \ref{prop_2.4}.
\end{proof}}

}

The set $\KI$ is  a $qm$-admissible wedge, that is  $a+b\in\KI$, $\lambda a\in \KI$ and  $x^*ax\in \KI$ for all  $a,b\in \KI$, $x\in\Ao$ and  $\lambda\geq0$. If, moreover, $\A$ has a unit $\id$ then $\id\in \KI$.

As usual, one can define an {order} on the   {\em real} vector space $\A_h=\big\{a \in \A:\, a=a^*\big\}$ by   $$
{a \leqm b \ \ \Leftrightarrow \ \  b-a \in \KI,\quad a,b\in\A_h.}
$$

\begin{proposition}\label{prop: 4.4.3}
Let $(\A,\A_0)$ be  quasi *-algebra with unit $\id$ and let $\M$ be a balanced subset of $\IA$. Then the following are equivalent:
\begin{itemize}
	\item[i)] $\M$ is sufficient;
	\item[ii)] $\KI\cap (-\KI)=\{0\}$.
\end{itemize}			
\end{proposition}\begin{proof} $i)\Rightarrow ii)$ Let $a\in\KI\cap (-\KI)$. Then $\vp(ax,x)=0$, for every $\vp\in \M$ and every $x\in\A_0$; hence, by Proposition \ref{prop_2.4}, 
and by the sufficiency of $\M$ we get $a=0$.\\ $ii)\Rightarrow i)$ Let us suppose by absurd that there exists $a\in \A$, $a\neq0$  such that $\vp(a,a)=0$, for every $\vp\in \M$. 
Then again by Proposition \ref{prop_2.4}, it follows that $\vp(ax,x)=0$ for every $x\in\A_0$ and every $\vp \in \M$; this means that $a\in \KI\cap (-\KI)=\{0\}$ a contradiction.  
\end{proof}

\section{Bounded and order bounded elements}\label{Sect_Bdd}
\begin{definition} \label{defn_mbounded}
Let $(\A,\Ao)$ be a quasi *-algebra with $\M\subset \IA$ sufficient. We  say that an element $a\in \A$ is $\M$-bounded if there exists $\gamma_a=\gamma_{a,\M}>0$ such that
$$ |\vp(ax,y)| \leq \gamma_a \vp(x,x)^{1/2}\vp(y,y)^{1/2}, \quad \forall \vp\in \M;\, \forall x,y \in \Ao.$$
If $a$ is $\M$-bounded, we put
\begin{align*}\|a\|\rnb=\inf\{\gamma_a>0: |\varphi(ax,y)| &\leq \gamma_a \vp(x,x)^{1/2}\vp(y,y)^{1/2},\\ &\quad\,\varphi\in\mathcal M, \,\,  x,y\in\Ao\}. \end{align*}
\end{definition}

{\begin{remark} \label{equalities}For future use, we notice, that, in general and regardless to the $\M$-boundedness of $a\in\A$, 
the following equalities hold:
\begin{align*}
	\Lambda_a &:=
	\inf\{\gamma_a>0: |\varphi(ax,y)| \leq \gamma_a \vp(x,x)^{1/2}\vp(y,y)^{1/2}, \,\varphi\in\mathcal M, \,  x,y\in\Ao\}\\ &= \sup\{ |\vp(ax,y)|; \vp\in \M, x,y\in \Ao, \vp(x,x)=\vp(y,y)=1 \}
	\\ &= \sup\{\|\pi_\vp(a)\|;\; \vp\in \M\}. \end{align*}

Moreover, if $a=a^*$,
$$\Lambda_a =\sup\{ |\vp(az,z)|; \vp\in \M, z\in \Ao, \vp(z,z)=1 \}.$$

\smallskip

The value of $\Lambda_a$ is finite if and only if $a$ is $\M$-bounded; by the definition itself, $\|a\|\rnb=\Lambda_a$. 
\end{remark}}
{
\begin{lemma}\label{prop_ equal norms} Let $a,b\in \A$ be {$\M$-bounded}.   Then 
	\begin{itemize}
		\item[(i)] $a^*$ is $\M$-bounded too, and $\|a^*\|\rnb=\|a\|\rnb$;
		\item[(ii)]$a+b$ is $\M$-bounded and $\|a+b\|\rnb\leq \|a\|\rnb + \|b\|\rnb$;
		\item[(iii)] $\alpha a$ is  $\M$-bounded, $\forall \alpha\in\mathbb{C}$;
		\item[(iv)] if $a\wmult b$ is well--defined, the product $a\wmult b$ is $\M$-bounded and
		$ \|a\wmult b\|\rnb\leq\|a\|\rnb\, \|b\|\rnb.$
	\end{itemize}	
\end{lemma} }

As in \cite[Proposition 4.18]{att_2010} one can prove
\begin{proposition} \label{prop_23} Let $a,b$ be $\M$-bounded elements of $\A$ and let $\vp\in \M$. Then, if $a\wmult b$ is well--defined, $\pi_\vp(a) \wmultt\pi_\vp(b)$ is also well--defined and $\pi_\vp(a\wmult b)=\pi_\vp(a) \wmultt\pi_\vp(b)$.
\end{proposition}

{\begin{remark} \label{rem_cstar}Remark \ref{equalities} and Proposition \ref{prop_23} imply that if $a$ is $\M$-bounded and $a^*\wmult a$ is well -defined, then $\|a^*\wmult a\|\rnb= (\|a\|\rnb)^2$.
	\end{remark}}

The notion of $\M$-positive element can be used to give a formally different  definition of bounded element. Let $a\in \A$, put $  {\Re (a)} =\frac{1}{2}(a+a^*)$,  $  {\Im (a)}=
\frac{1}{2i}(a-a^*)$. Then both $ \Re(a),\Im (a)  \in{\A_h}$ and $a=\Re(a)+i\Im (a).$ 
\begin{definition} Let $(\A,\Ao)$ be a quasi *-algebra.  The element $a\in \A$ is said {\em $\KI$-bounded}  if there exists  {$ \gamma \geq 0$}, such that
\begin{equation}\label{eqn_ordbdd}
\left\{ \begin{array}{l} \pm \vp(\Re(a)x,x) \leq \gamma \vp(x,x) \\ \pm \vp(\Im(a)x,x)\leq \gamma \vp(x,x) \end{array} \right. , \; \forall \vp\in \M, x\in \Ao.
\end{equation} 

If $(\A,\Ao)$ is unital, then we can rewrite \eqref{eqn_ordbdd}, more syntetically, as $$ \pm \Re(a) \leqm \gamma \id, \qquad \pm \Im(a)\leqm \gamma \id.$$
We denote by $\bdd$ the set of all $\KI$-bounded elements of $\A$.
\end{definition}

{	As in \cite{Frag2} the following result holds true:
\begin{proposition}\label{prop: simil 6.4.2} 
	The couple {$(\bdd,\bdd \bigcap \A_0)$ is a quasi *-algebra}, hence, in particular\begin{enumerate}
		\item $\alpha a+\beta b\in \bdd$ for any $\alpha,\beta\in\mathbb{C}$ and $a,b\in\bdd$;
		\item $a\in\bdd\Leftrightarrow a^*\in\bdd$;
		\item $a\in\bdd$, $x\in\bdd\bigcap\A_0\Rightarrow xa\in \bdd$;
		\item $x\in\bdd\bigcap\A_0\Leftrightarrow xx^*\in\bdd\bigcap\A_0$.
	\end{enumerate}{In particular, if $(\A,\A_0)$ has a unit, then also $(\bdd,\bdd \bigcap \A_0)$ has a unit.}
\end{proposition}	}

\begin{theorem}\label{thm: simil 6.4.5} 
Let $(\A,\Ao)$ be a quasi *-algebra  with sufficient $\M\subset \IA$. Then the following are equivalent:	\begin{itemize}
	\item[$i)$] 	$a\in\bdd$; 		
	\item[$ii)$]  $a$ is $\M$-bounded; i.e., {\rm (Definition \ref{defn_mbounded})} there exists $\gamma_a>0$ such that $$|\vp(ax,y)|\leq\gamma_a\vp(x,x)^{1/2}\vp(y,y)^{1/2}, \,\forall x,y\in \A_0,$$ for every 
	$\vp\in\M;$ 
	\item[$iii)$] there exists $\gamma'_a>0$ such that
	$$\vp(ax,ax)\leq {\gamma'_a} \vp(x,x), \,\forall x\in \Ao,$$ for every $\vp\in\M.$ 
	\item[$iv)$] there exists $\gamma''_a>0$ such that $$|\vp(ax,x)|\leq\gamma''_a\vp(x,x), \,\forall x\in \Ao,$$ for every $\vp\in\M$. 
\end{itemize}
\end{theorem}
\begin{proof} We prove it for symmetric elements.\\$i)\Rightarrow ii)$ 
It is clear that if $a=a^*\in \bdd$, then there exists $\gamma>0$ such that $$|\vp(ax,x)|\leq \gamma \vp(x,x),\,\forall \vp\in\M, \forall x\in\A_0$$ hence  $$\sup\{|\vp(ax, x)|; \; \vp \in \M, x \in \Ao, \vp(x,x)=1\}<\infty.$$ 
From Remark \ref{equalities} it follows that
$\|a\|\rnb<\infty$ i.e. $a$ is $\M$-bounded.
\\$ii) \Rightarrow iii)$ 
Assume that $a\in\A$ is $\M$-bounded. If $\vp \in \M$, denote by $\pi_\vp$ the corresponding GNS representation. Then,
\begin{align*}
	|\ip{\pi_\vp(a)\lambda_\vp(x)}{\lambda_\vp(y)}|	&= |\varphi(ax,y)| \leq \|a\|\rnb \vp(x,x)^{1/2}\vp(y,y)^{1/2}\\ &
	=\|a\|\rnb\,\|\lambda_\vp(x)\|\, \|\lambda_\vp(y)\|, \quad\, \forall x,y\in\A_0.
\end{align*}
This implies that, for every $\vp \in \M$, the operator $\pi_\vp(a)$ is bounded and $\|\pi_\vp(a)\|\leq \|a\|\rnb$. Hence, 
\begin{align}\label{eq:ineq 2}\nonumber
	\vp(ax,ax)^{1/2}&= \|\pi_\vp(a)\lambda_\vp(x)\|\\&\leq \|a\|\rnb\|\lambda_\vp(x)\|=\|a\|\rnb\vp(x,x)^{1/2},\quad \, \forall x,y\in\A_0.
\end{align}
$iii)\Rightarrow iv)$ Suppose that $a$ satisfies $iii)$. Let $\vp\in \M$ and $x\in \Ao$. Then
$$|\vp(ax,x)|\leq\vp(ax,ax)^{1/2} \vp(x,x)^{1/2}\leq {\gamma'_a}^{1/2}\vp(x,x).$$ 
\\$iv)\Rightarrow i)$ It is straightforward.
\end{proof}

\begin{remark} By the previous theorem we also deduce the following equalities for the norm of a $\M$-bounded element $a$ (see also Remark \ref{equalities}):
\begin{align*}
\|a\|\rnb &=\inf\{\gamma>0: \varphi(ax,ax)\leq\gamma^2\varphi(x,x), \,\varphi\in\mathcal M, \,  x\in\Ao\}
\\
&= \sup\{\varphi(ax,ax)^{1/2}: \varphi\in \mc M, \,x\in
\Ao,\, \varphi(x,x)=1\}.
\end{align*}
\end{remark}

{In view of Theorem \ref{thm: simil 6.4.5}, we adopt the notation $\bdds$ for the set of either $\M$-bounded or $\KI$-bounded elements; i.e., we put $\bdds=\bdd$.  }

\begin{definition} Let $(\A,\Ao)$ be a quasi *-algebra and let $\M\subset \IA$.
We say that a *-representation $\pi$ of $\A$ is $\M$-{\em regular} if for every $\xi\in\D_\pi$, the vector form $\vp_\xi$ defined	by
\begin{equation}\label{eqn_vpxi}\vp_\xi (a, b):= \ip{\pi(a)\xi}{\pi(b)\xi},\quad a,b \in\A \end{equation}
is a form in $\M$.\\
In particular, if $\M= \IA$, $\pi$ is said to be {\em regular} \cite{FT_book}.
We denote by $\mrep$ the set of $\M$-regular *-representations of $(\A,\Ao)$. If $\M= \IA$ we denote it simply by $\irep$.
\end{definition}

\begin{remark} If $\M\subset \IA$ is balanced, then for every $\vp\in\M$ the *-representation $\pi_\vp^\circ$ is $\M$-regular, see \cite[Proposition 2.4.16]{FT_book}.\\ 
\end{remark}

\begin{proposition}\label{prop: positive elem}Let $(\A,\Ao)$ be a quasi *-algebra with unit $\id$ and let $a\in \A$.

If $\pi(a)\geq0$ for every  *-representation $\pi$  of $(\A,\A_0)$, then $a\in\KII$. 
Conversely, if $\M\subset \IA$ and $a\in\KI$, then $\pi(a)\geq0$ for every $\M$-regular *-representation $\pi$  of $(\A,\A_0)$.
\end{proposition}
\begin{proof}If $\pi(a)\geq0$ for every  *-representation $\pi$  of $(\A,\A_0)$, then  for every $\vp\in\IA$ and every $x\in\A_0$, $$\vp(ax,x)=\ip{\pi_\vp(a)\lambda_\vp(x)}{\lambda_\vp(x)}\geq0.$$ Hence, $a\in\KII$.\\Conversely, let $\pi$ be a $\M$-regular *-representation. Then, for every $\xi\in\D_\pi$,
the vector form $\vp_\xi(a,b)=\ip{\pi(a)\xi}{\pi(b)\xi}$, with $a,b\in\A$, belongs to $\M$. Thus, from $a\in\KI$,  it follows that $\vp_\xi(a,\id)=\ip{\pi(a)\xi}{\xi}\geq0$, for every $\xi\in\D_\pi$. Hence $\pi(a)\geq0$.
\end{proof}

\begin{remark} 
The first implication is also true if we consider only the GNS representations constructed from the forms $\vp\in \IA$; if $\vp\in \M\subset \IA$ {and every   $\pi_\vp(a)\geq0$,} then $a\in \KI$.
\end{remark}

\begin{proposition}\label{bounddn of operators} Let $(\A,\Ao)$ be a quasi *-algebra { with unit $\id$ and}
with sufficient $\M\subset \IA$.  Then,

\begin{itemize}
\item[{(i)}]	if $a\in\bdds$ then  $\pi(a)$ is a bounded operator  for every $\pi\in\mrep$ and 
$\|\pi(a)\|\leq \|a\|\rnb$;
\item[(ii)] 
if  $\pi(a)$ is a bounded operator for every $\pi\in\mrep$ and\\ \mbox{$\displaystyle\sup\{\|\pi(a)\|; \pi \in \mrep\}<\infty$}, then $a\in\bdds$. 
\end{itemize}
\end{proposition}
\begin{proof}$(i)$ By Theorem \ref{thm: simil 6.4.5}, $a\in\bdds$ implies $\vp(ax,ax)^{1/2} \leq \|a\|\rnb \vp(x,x)^{1/2}$ $\forall \vp \in \M;\, x\in \Ao$.
If $\pi$ is $\M$-regular, for every $\xi\in\D_\pi$, $\vp_\xi \in \M$ where $\vp_\xi(a,b)= \ip{\pi(a)\xi}{\pi(b)\xi}$.
Then,
$$\|\pi(ax)\xi\|= \vp_\xi(ax,ax)^{1/2}  \leq \|a\|\rnb \vp_\xi(x,x)^{1/2}  = \|a\|\rnb \|\pi(x)\xi\|.$$
The quasi *-algebra is supposed to be unital and also  $\pi(\id)=I_{\D_\pi}$. Then
$$\|\pi(a)\xi\| \leq \|a\|\rnb \|\xi\|, \quad \forall \xi \in \D_\pi.$$
Hence,  $\|\pi(a)\|\leq \|a\|\rnb$.\\
$(ii)$  Put $\displaystyle \gamma:= \displaystyle\sup\{\|\pi(a)\|; \pi \in \mrep\}.$  By hypothesis 
$$ \|\pi_\vp^\circ(a)\lambda_\vp(x)\| \leq \gamma \|\lambda_\vp(x)\|, \quad \forall \vp\in \M, x\in\A_0$$ i.e. $$\vp(ax,ax)\leq {\gamma}^2 \vp(x,x), \, \quad \forall \vp\in \M, x\in\A_0$$ and by Theorem \ref{thm: simil 6.4.5} it is equivalent to say that  $a\in\bdds$. 
\end{proof}

\medskip

\begin{remark} \label{rem_414}Let $\vp\in \M$ and denote, as before, by $\pi_\vp$ the corresponding GNS representation. If $a\in \bdds$, then $\pi_\vp(a)$ is a bounded operator. Indeed, for every $x\in \Ao$,
$$\|\pi_\vp(a)\lambda_\vp(x)\|^2 =\vp(ax,ax)\leq (\|a\|\rnb)^2 \vp(x,x)= (\|a\|\rnb)^2\|\lambda_\vp(x)\|^2,$${ regardless of whether  $\pi_\vp$ is $\M$-regular or not}. 
\end{remark}
The following additional condition will be used:
\begin{itemize}
\item[(C)] if $a, b\in \bdds$ then
there exists a unique $c\in \A$, such that $\pi_\vp(a)\mult\pi_\vp(b) = \pi_\vp(c)$, for every $\vp \in \M$.
\end{itemize}
\begin{theorem}\label{thm: normed *-algebra}  Let $(\A,\Ao)$ be a quasi *-algebra.
Let $\M\subset \IA$ be  sufficient and  assume that condition (C) holds. 
Then, $\bdds$  is a normed *-algebra with the weak multiplication $\wmult$ and the norm $\|\cdot\|\rnb$. 
\end{theorem}
\begin{proof} As we have seen until now, $\bdds$ is a normed space such that if $a\in\bdds$ then  $a^*$ is $\M$-bounded  and $\|a^*\|\rnb=\|a\|\rnb$ and, whenever $a\wmult b$ is well--defined, the product $a\wmult b$ is $\M$-bounded and $ \|a\wmult b\|\rnb\leq\|a\|\rnb\, \|b\|\rnb$. Now, if  $\vp\in \M$ and $a, b \in \bdds$, the operator $\pi_\vp(a)\mult\pi_\vp(b)$ is well--defined since, 
by Remark \ref{rem_414},  $\pi_\vp(a)$ and $\pi_\vp(b)$ are bounded operators; of course, $\pi_\vp(a)\mult\pi_\vp(b)$ is also bounded. Thus, by (C),
there exists a unique $c \in \bdds$, such that $\pi_\vp(a)\mult\pi_\vp(b) = \pi_\vp(c)$, for every $\vp\in \M$.
Hence, for all  $\vp\in \M$ and $x, y \in \A_0$, we have
\begin{align*}
\vp(bx, a^*
y) &= \ip{\pi_\vp(b)\lambda_\vp(x)}{\pi_\vp(a^*)\lambda_\vp(y)}= \ip{\pi_\vp(a)\mult\pi_\vp(b)\lambda_\vp(x)}{\lambda_\vp(y)}\\
&= \ip{\pi_\vp(c)\lambda_\vp(x)}{\lambda_\vp(y)}
= \vp(cx, y).
\end{align*}
Thus $c=a\wmult b$ is well--defined and  is $\M$-bounded by {Lemma \ref{prop_ equal norms}}. This completes  the proof.
\end{proof}

{\begin{proposition}  Let $(\A,\Ao)$ be a quasi *-algebra with unit $\id$.
Let $\M\subset\IA$ be balanced and denote by $R^\M(\A)$ the intersection
of the kernels of all $\M$-regular *-representations of $\A$ on some Hilbert space. Then $$R^\M(\A)=\{a\in\A|\, \vp(a,a)=0, \,\forall \vp\in \M\}.$$
\end{proposition}\begin{proof} For every $\vp\in\M$, the GNS representation $\pi^\circ_\vp$ is $\M$-regular, then, if $a\in R^\M(\A)$, it is $\pi^\circ_\vp(a)=0$, hence $\vp(a,a)=\|\pi^\circ_\vp(a)\xi_\vp\|=0$.\\Conversely, if $a\in \A$ is such that $\vp(a,a)=0$, $\forall \vp\in \M$, since $\M$ is balanced, it is $\vp^x(a,a)=0$ for all $x\in\A_0$ and for all $\vp\in\M$, hence
$$\vp^x(a,a)=\vp(ax,ax)=\| \pi^\circ_\vp(a)\lambda_\vp(x) \|^2=
0$$and by the density of $\lambda_\vp(\A_0)$ in $\H_\vp$ this implies that $\pi^\circ_\vp(a)=0$. Now, let $\pi$ be a $\M$-regular *-representation of $\A$, then for every $\xi\in\D_\pi$ the form $\vp_\xi\in\M$ and by what we have seen before it is $\pi^\circ_{\vp_\xi}(a)=0$. For every $\xi\in\D_\pi$,  there exists a cyclic vector $\eta$ for the GNS representation $\pi^\circ_{\vp_\xi}$ such that $\pi^\circ_{\vp_\xi}$  $$\|\pi(a)\xi\|^2=\|\pi^\circ_{\vp_\xi}(a)\eta\|^2=0$$this implies that $\pi(a)=0$. By the arbitrariness of the $\M$-regular *-representation $\pi$ of $\A$, it follows that $a\in R^\M(\A)$. This concludes the proof.
\end{proof}
\begin{remark}
The set $R^\M(\A)$ is clearly a sort of *-radical; however its nature is purely algebraic  here.
\end{remark}}

\section{Topologies defined by families of sesquilinear forms}\label{Sect_Topologies}
The properties we have discussed in the previous section are all of pure algebraic nature; but families of sesquilinear forms of $\IA$ can be used to define, in rather natural way, topologies on $(\A,\Ao)$.  Our next goal is in fact to define in $\A$ topologies that mimick the uniform topologies of families of operators.

{Throughout this section we will suppose that $(\A,\Ao)$ is a quasi *-algebra with unit $\id$ and that $\M\subset \IA$ is a sufficient set of forms.}
Then $\M$ defines the topologies $\taum_w, \taum_s, \taum_{s^*}$  generated, respectively, by the following families of seminorms: 
\begin{itemize}
	\item[$\taum_w $:]
	$\quad a\mapsto |\varphi(a x,y)|$, $\quad a \in \A, \ \varphi\in\M, \ x,y\in\Ao$;
	\item[$\taum_s $:]
	$\quad a\mapsto \varphi(a,a)^{1/2}$, $\quad a \in \A, \ \varphi\in\M$;
	\item[$\taum_{s^*} $:]
	$\quad a\mapsto \max\big\{\varphi(a,a)^{1/2},\varphi(a^*,a^*)^{1/2}\big\}, \quad a \in \A, \ \varphi\in\M$.
\end{itemize}

{\begin{definition} Let $\F$ be a subset of $\M$. We say that $\F$ is {\em bounded} if $$\sup_{\vp\in\F} \vp(a,a)<\infty,\quad \forall a\in \A.$$ \end{definition}}
The family $\mathfrak{F}$ of bounded subsets   of forms in $\M$ has the following properties:

\begin{itemize}
	\item[(a)]  $\displaystyle \bigcap_{n\in\mathbb{N}}\F_n\in\mathfrak{F}, \quad \F_n \in \mathfrak{F}$;
	\item[(b)] $\F\cup\mathcal{G}\in\mathfrak{F}, \quad \F, \mathcal{G} \in\mathfrak{F}$.
\end{itemize}  If  $\F\in\mathfrak{F}$, we put
$$ p^\F(a):= \sup_{\vp\in\F} \vp(a,a)^{1/2}, \quad a\in \A.$$

\begin{lemma}\label{lemma1} Let $\F\in\mathfrak{F}$. Then,
	\begin{itemize}
		\item[(a)] $p^\F$ is a seminorm on $\A$; 
		\item[(b)] the set $\F^x=\{\vp^x, \vp \in \F\}$, $x\in \Ao$, is bounded;
		\item[(c)] for every $x\in \Ao$,
		$$	p^\F (ax) = p^{\F^x}(a), \quad\forall a\in \A.$$
	\end{itemize}	
\end{lemma}

\begin{proof} As for (c), we have
	$$	p^\F (ax)=  \sup_{\vp\in\F} \vp(ax,ax)^{1/2}= \sup_{\vp\in\F} \vp^x(a,a)^{1/2}= \sup_{\psi\in\F^x} \psi(a,a)^{1/2}=p^{\F^x}(a).$$
\end{proof}

Since $\M$ is sufficient,
then $\{p^\F; \F\in \mathfrak{F}\}$  is a separating family of seminorms; thus it defines on $\A$ a Hausdorff locally convex topology which we denote by $\ltaum^\mathfrak{F}$.

Let us assume that $(\A, \Ao)$ has a unit $\id$.
If  $\F\in\mathfrak{F}$ we define
$$p_\F(a):= \sup_{\vp\in\F}|\vp(a,\id)| ,\quad a\in \A.$$

Then $$p_\F(a)\leq \gamma_\F p^\F(a), \quad \forall a\in \A, $$ with $\gamma_\F=\sup_{\vp\in\F}\vp(\id,\id)^{1/2}$,
and the following hold: {
	\begin{align*} &p_\F(a^*)=p_\F(a), \quad \forall a\in \A; \\
		& 	p_\F(ax)\leq p^\F(x) p^\F(a^*), \quad \forall a\in \A, x\in \Ao; \\
		&{p_\F(a^*\wmult a)=p^\F(a)^2, \; \forall a\in \A \mbox{ such that $a^*\wmult a$ is well--defined}.}				
	\end{align*}
}

By {$\ltaum_\mathfrak{F}$} we will denote the locally convex topology  on $\A$ generated by the family of seminorms $\{p_\F; \F\in \mathfrak{F}\}$ (to simplify notations, we do not mention explicitly the dependence on $\M$). Note that $\ltaum_\mathfrak{F}$ need not be Hausdorff, in general.

{\begin{remark} We notice that if $a^*\wmult a$ is well--defined and $a^*\wmult a=0$ then $p^\F(a)=0$ for every bounded set $\F{\in\FF}$ and therefore $a=0$.	\end{remark} 	}

{	\begin{proposition} Let $\M$ be sufficient and suppose that $\ltaum_\mathfrak{F}= \ltaum^\mathfrak{F}$. Then $\A[\ltaum^\mathfrak{F}]$ is a locally convex space with the following properties: 
		\begin{itemize}
			\item[(i)] the involution $a\mapsto a^*$ is continuous;
			\item[(ii)] for every bounded set $\F{\in\FF}$ there exists a bounded set ${\mc G }{\in\FF}$, 
			$$ p_\F(ax)\leq p^{\mc G }(a)p^{\mc G }(x), \quad \forall a \in \A, x\in \Ao;$$ which implies that the left- and right multiplications are jointly continuous. 	\end{itemize}
		In particular, if $\Ao$ is $ \ltaum^\mathfrak{F}$-dense in $\A$ then $(\A[\ltaum^\mathfrak{F}], \Ao)$ is a locally convex quasi *-algebra.
		
	\end{proposition}	
}	
In general,	the involution $a\mapsto a^*$ in not continuous for $\ltaum^\mathfrak{F}$. To circumvent this problem,  we define the topology $\ltaum^\mathfrak{F}_*$ generated by the family of seminorms		
$$p^\F_*(a) =\max\left\{p^\F(a), p^\F(a^*)\right\}, \quad a\in \A,\, \F\in \mathfrak{F}.$$
{ Clearly $\ltaum_\mathfrak{F}\preceq \ltaum^\mathfrak{F}\preceq \ltaum^\mathfrak{F}_* $ and, if $\ltaum_\mathfrak{F}= \ltaum^\mathfrak{F}$, then $\ltaum_\mathfrak{F}= \ltaum^\mathfrak{F}= \ltaum^\mathfrak{F}_* $.}

\medskip	
{ Let $(\A,\Ao)$ be a quasi *-algebra and suppose that the set $\M$ is sufficient. It is clear that every $\vp \in \M$ is automatically continuous for $\taum_s$ and for any finer topology such as $\taum_{s^*}$, $\ltaum^\mathfrak{F}_*$  or $\ltaum^\mathfrak{F}$.

	Our next goal is to investigate the properties of $(\A,\Ao)$ when $\A$ is endowed with one of the topologies defined by the family $\M$ defined above.

 We could wonder whether	$(\A[\tau^\mathfrak{F}_*], \Ao)$ is a locally convex quasi *-algebra.  
The left- and right multiplications by an element $x\in \Ao$ are continuous if we make an additional assumption:
	let us suppose that for every $x\in \Ao$ 	there exists $\gamma_x>0$ such that 
	\begin{equation} \label{eqn_leftbou}
		\vp(xa,xa)\leq \gamma_x \vp(a,a), \quad \forall \vp \in \M , \forall a \in \A.
	\end{equation} 
	{By \eqref{eqn_leftbou} it follows that every $x\in \Ao$ is $\M$-bounded and
		for every bounded subset $\F\subset \M$,}
	$$p^\F(xa)\leq \gamma_x p^\F(a), \quad \forall a \in \A.$$
	This inequality, together with (c) of Lemma \ref{lemma1} and the continuity of the involution, implies that,  for every $x\in \Ao$, the maps $a\mapsto ax$, $a\mapsto xa$ are $\tau^\mathfrak{F}_*$-continuous. The *-algebra $\Ao$ is not $\tau^\mathfrak{F}_*$-dense in $\A$ in general, hence in order to get a locally convex quasi *-algebra with topology $\tau^\mathfrak{F}_*$ we could take as $\A$ the completion $\widetilde{\Ao}^{\tau^\mathfrak{F}_*}$. 
Now we prove the following
\begin{lemma}
	Let $(\A,\Ao)$ be a quasi *-algebra.
	Assume that $\M$ is sufficient and directed upward w.r.to the order 
	$$\vp\leq \psi \;\Leftrightarrow\; \vp(a,a)\leq \psi(a,a), \quad \forall a\in \A.$$
	Then $\Ao$ is dense in $\A[\tau_{s}^\M]$.
\end{lemma}
\begin{proof}
	Let us begin with proving  that given $a\in \A$ we can find a net $\{x_\alpha\}\subset \Ao$ such that $\vp(x_\alpha-a, x_\alpha-a)\to 0$ for every $\vp\in \M$. { Again we  put $\vp[a]:=\vp(a,a)$, $\vp \in \M$, $a\in \A$.}\\
	Since $\vp\in \M$, $\lambda_\vp(\Ao)$ is dense in $\H_\vp$   (with $\H_\vp$ defined as in  Proposition \ref{idix}). This implies that, for every $a\in \A$, there exists a sequence $\{x_n^\vp\}$ such that $\lambda_\vp(x_n^\vp -a)\to 0$ or, equivalently $\vp[x_n^\vp -a]\to 0$. Then 
	$$\forall n\in {\mb N}, \quad \exists n_\vp\in{\mb N} :\; \vp[x_{n_\vp}^\vp -a]<\frac1n.$$
	If $\vp, \psi \in \M$, we define $(\vp,n_\vp)\leq (\psi, n_\psi)$ if $\vp\leq \psi$ and $n_\vp \leq n_\psi$. Since $\M$ is directed, $\{(\vp,n_\vp)\}$ is directed and $	\{x_{(\vp, n_\vp)}\}$  is a net, with $x_{(\vp, n_\vp)}:= x_{n_\vp}^\vp$. We prove that, for every $\psi \in \M$ $\psi[x_{n_\vp}^\vp-a]\to 0$. Indeed, let $\epsilon >0$ and $n \in {\mb N}$ such that $\frac1n <\epsilon$. Then if $ (\vp,n_\vp) \geq (\psi, n_\psi)$
	$$\psi[x_{n_\vp}^\vp-a]\leq \vp[x_{n_\vp}^\vp-a]<\frac1n <\epsilon.$$
	This proves that $\Ao$ is dense in $\A[\tau_{s}^\M]$. 
\end{proof}

	The representation $\pi^\circ_\vp$ is $\tau^\mathfrak{F}$-continuous. Indeed, if ${\mc F}$ is any bounded subset of $\M$ containing $\vp$,
	\begin{equation*} 
		\|\pi^\circ_\vp(a)\lambda_\vp(x)\|=\vp(ax,ax)^{1/2}\leq  p^{\F}(ax) \leq  p^{\F^x}(a),\quad \forall a\in \A; \, x\in \Ao \end{equation*}
	as in Lemma \ref{lemma1}.
	
	\begin{proposition}
		Let $(\A,\Ao)$ be a quasi *-algebra with sufficient $\M\subset \IA$. If $\A$ is $\tau^\M_{s^*}$-complete, then $\A$ is also $\tau^\FF_*$-complete.
	\end{proposition}
	\begin{proof}
		Let $\{a_\alpha\}$ be a $\tau^\FF_*$-Cauchy net. Since $\tau^\M_{s^*}\preceq \tau^\FF$, there exists $a\in \A$ such that $a=\tau^\M_{s^*}-\lim_\alpha a_\alpha$. From the Cauchy condition, for every $\epsilon >0$ and every bounded set $\F$ there exists $\overline{\alpha}$ such that
		$$\max\{\vp (a_\alpha - a_{\alpha'},a_\alpha - a_{\alpha'}), \vp (a^*_\alpha - a^*_{\alpha'},a^*_\alpha - a^*_{\alpha'}) \}< \epsilon, \; \forall \vp \in \F, \alpha, \alpha' > \overline{\alpha}.$$
		Then taking limit over $\alpha'$,
		$$\max\{\vp (a_\alpha - a,a_\alpha - a), \vp (a^*_\alpha - a^*,a^*_\alpha - a^*) \}\leq  \epsilon, \; \forall \vp \in \F, \alpha> \overline{\alpha}.$$
		Therefore, $\A$ is $\tau^\FF_*$-complete. 
	\end{proof}
	
	\begin{theorem} \label{thm struct} Let $\M$ be sufficient and let property $(C)$ hold too. If $\A$ is $\tau^\M_{s^*}$-complete, then $\bdds$  is a  C*-algebra with the weak multiplication $\wmult$ and the norm $\|\cdot\|\rnb$. 
	\end{theorem}
	\begin{proof} 
		By Theorem \ref{thm: normed *-algebra} {and Remark \ref{rem_cstar}} we only  need to prove the completeness of $\bdds$.
		Let $\{a_n\}\subset\bdds$ be a Cauchy sequence with respect to the norm  $\|\cdot\|\rnb$. Then, $\{a_n^*\}$ is $\|\cdot\|\rnb$-Cauchy too. By \eqref{eq:ineq 2},  $$\vp((a_n - a_m)x, (a_n - a_m)x)
		\leq	(\|a_n - a_m\|\rnb)^2 {\vp} (x,x), \, \forall \vp\in\M, \forall x\in\A_0$$
		and $$\vp((a_n^*- a_m^*)x,(a_n^*- a_m^*)x)
		\leq	(\|a_n^*- a_m^*\|\rnb)^2 {\vp} (x,x), \, \forall \vp\in\M, \forall x\in\A_0.$$ Therefore both
		$\vp((a_n - a_m)x, (a_n - a_m)x)\to 0$ and	$\vp((a_n^*- a_m^*)x,(a_n^*- a_m^*)x)\to 0,$ as  $n,m\to\infty.$
		
		This is in particular true when $x=\id$ hence $\{a_n\}$ is also Cauchy with respect to $\tau_{s^*}^\M$ and since $\A$ is $\tau_{s^*}^\M$-complete, 
		there exists $a \in \A$ such that $a_n \stackrel{\tau_{s^*}^\M}{\to} 
		a$. The limit $a\in\bdds$; indeed, for every $\vp\in\M$:
		\begin{align*}
			|\vp(ax,x)|^2&\leq\vp(ax,ax)\vp(x,x)=\vp(x,x)\limsup_{n\to\infty}\vp(a_nx,a_nx)\\&\leq{\sup_{n\in \mb{N}}\|a_n\|\rnb\vp(x,x)^2}
		\end{align*} 
		Since $\{a_n\}$ is Cauchy w.r.to the norm $\|\cdot\|\rnb$, for every $\epsilon > 0$ there exists $n_\epsilon \in \mathbb{N}$, such that ${\|a_n-a_m\|\rnb} < \epsilon^{1/2}$, for all
		$n,m > n_\epsilon$. This implies that
		$\vp((a_n - a_m)x, (a_n - a_m)x)< \epsilon\vp(x, x)$, $\forall \vp \in \M, \forall x \in \A_0$, forall $n,m>n_\epsilon$.
		Then, if we fix $n > n_\epsilon$ and let $m\to\infty$, we obtain
		$\vp((a_n - a)x, (a_n - a)x){\leq} \epsilon\vp(x, x), \forall \vp \in \M, \forall x \in \A_0.$ 
		This implies that $\bdds$ is  complete w.r.to the norm $\|\cdot\|\rnb$.
	\end{proof}

	To conclude, let us suppose that $\M\subset \IA$ is balanced. We pose the question: {\em under what conditions is $\M$ also sufficient?} Let us consider a locally convex quasi *-algebra $(\A[\tau],\Ao)$ and choose $\M = \PA{\tau}$. This set is certainly balanced, but it is not necessarily sufficient. This property can be characterized (by negation) by the following
	{\begin{proposition}\label{prop 2121} 
			Let $(\A[\tau],\Ao)$ be  a locally convex quasi *--algebra with unit $\id$. For an element $a\in \A$ the following statements are equivalent:
			\begin{itemize}
				\item[{\em (i)}] $a \in {\rm Ker\, }\pi$ for every  strongly-continuous {(i.e., ${\sf t}_s$-continuous)} qu*--representation $\pi$ of $(\A[\tau],\Ao)$;
				\item[{\em (ii)}]$\vp(a,a)=0$, for every $\vp \in \PA{\tau}$;
				\item[{\em (iii)}]$p^\F(a)=0$, for every bounded subset $\F$ of $\PA{\tau}$.
			\end{itemize}
		\end{proposition}

		\section{Locally convex quasi GA*-algebras} \label{Sect_GA}
		The discussion of the previous sections suggests  the following definition (which strengthen an analogous one for partial *-algebras \cite[Definition 4.26]{att_2010}).
		
		\begin{definition}
		\label{def:wbehaved} Let $(\A, \Ao)$ be a quasi *-algebra.
		Let $\M$ be a family of forms of $\IA$.  We say that $\M$ is {\em strongly well-behaved} if
		\begin{itemize}
			\item[({\sf wb}$_1$)] $\M$ is sufficient;
			\item[({\sf wb}$_2$)] every $x\in \Ao$ is $\M$-bounded;
			\item[({\sf wb}$_3$)]  condition (C) holds;
			\item[({\sf wb}$_4$)] $\A$ is $\ltaum^\mathfrak{F}_*$-complete.
		\end{itemize}
		\end{definition}
		\begin{definition} Let $(\A[\tau],\Ao)$ be a locally convex quasi *-algebra. We say that $(\A[\tau],\Ao)$ is a
		{\em locally convex quasi GA*-algebra} if there exists $\M\subset \IA$ which is strongly well-behaved and  $\tau$ and $\ltaum^\mathfrak{F}_*$ are equivalent (in symbols $\tau \approx \ltaum^\mathfrak{F}_*$).
		\end{definition}
		
		\begin{example} Let us consider the quasi *-algebra  $(\LDH, \LDb)$ of Section \ref{Sect_Prelim}. Assume that $\LDH$ is endowed with the topology $\sft^u_*$ and denote by $\LDH_u$ the $\sft^u_*$-closure of $\LDb$ in $\LDH$, then \\ \mbox{$(\LDH_u[\sft^u_*], \LDb)$}
		is a locally convex quasi *-algebra. Let us take as $\M$ the space consisting of the restrictions to $\LDH_u$ of the   ${ \sft^u_*}$-continuous ips-forms on $(\LDH, \LDb)$. {We will see that $\M$ is strongly well-behaved and $\sft^u_* \approx \ltaum^\mathfrak{F}_*$: this makes of  $(\LDH_u[\sft^u_*], \LDb)$ a locally convex quasi GA*-algebra.} Due to the ${\sft^u_*}$-density of $\LDb$ in $\LDH$, we can identify  $\M$ with the space of all ${\sft^u_*}$-continuous ips-forms on $(\LDH_u, \LDb)$. This implies \cite[Theorem 3.10]{att_2010} that every $\psi \in \M$ can be written as follows $\psi(A,B)=\sum_{i=1}^n\ip{A\xi_i}{B\xi_i}$ $A,B\in\LDH_u$, for some vectors $\xi_1,...,\xi_n\in \D$. 
		Hence, the set of $\M$-bounded elements coincides with the set $\LDH_b$ of all bounded operators of  $\LDH$, which can be identified with the C*-algebra $\BH$ of all bounded operators in $\H$.
		
		These facts allow us to conclude easily that $\M$ is strongly well-behaved. In particular, we notice that ({\sf wb}$_3$) holds since if $A,B\in \LDH_b$ then the multiplication $\wmult$ (see \eqref{def-3.1.30}) is well--defined and coincides with the weak multiplication $\mult$ of operators (see \eqref{eqn_wmult}): $A\wmult B= A\mult B$ is certainly well--defined; then if $\vp \in \M$, we have $\pi_\vp(A)\wmultt \pi_\vp(B)=\pi_\vp(A\mult B)=\pi_\vp(A\wmult B)$, by definition of *-representation.

		\end{example}

		\begin{example} \label{ex_ellepi} Let $K$ denote a compact subset of the real line with $m(K)>0$, where $m$ denotes the Lebesgue measure. Then the pair $(L^p(K, m), C(K))$, where $C(K)$ denotes the C*-algebra of continuous functions on $K$, is a Banach quasi *-algebra. Let $\M$ be the space of all jointly continuous ips-forms on $(L^p(K,m), C(K))$. Then, as shown in \cite[Example 3.1.44]{FT_book}, if $p\geq 2$, $\M$ can be completely described by functions of $L^{s}(K,m)$, where $s=\frac{p}{p-2}$ $(\frac10=\infty)$, in the following sense:
		$$ \vp \in \M \Leftrightarrow \exists w\in L^s(K,m), \, w\geq 0: \; \vp(f,g)=\int_K f\overline{g}w dm, \;\forall f,g \in L^p(K, m).$$
		
		For this reason we identify $\M$ with $L^s(K,m)$. With this in mind,
		\begin{itemize}
			\item[(a)] a subset $\F$ of $\M$ is bounded, if and only if it is contained in a ball centered at $0$ in $L^s(K,m)$;
			\item[(b)] the topology $\tau^\FF$ (which equals $\tau^\FF_*$, in this case) is normed and the norm coincides with $\|\cdot\|_p$, since 
			$$\sup_{\|w\|_s=1}\int_K |f|^2w dm= \||f|^2\|_{p/2}=\|f\|_p^2.$$
			\item[(c)] The topology $\tau_\FF$ is also a norm topology and the norm coincides with $\|\cdot\|_{p/2}$;
			\item[(d)] The set of $\M$-bounded elements is the C*-algebra $L^\infty(K,m)$.
		\end{itemize}
		In conclusion, $(L^p(K,m), L^\infty(K,m))$ is a Banach quasi GA*-algebra.
		\end{example}
		
		\begin{example} The space $L^p_{\rm loc}({\mb R},m)$ of all (classes of) measurable functions on ${\mb R}$ such that the restriction $f_{\upharpoonright K}$ of $f$ to $K$ is in $L^p(K,m)$, for every compact subset $K\subset {\mb R}$, behaves similarly to the case discussed in Example \ref{ex_ellepi}. The main difference consists, of course, in the fact that we will not deal with norm topologies. More precisely, let us consider the pair $(L^p_{\rm loc}({\mb R},m), C_b({\mb R}))$ (where $C_b({\mb R})$ denotes the continuous bounded functions on ${\mb R}$), which is, as it is easy to check, a quasi *-algebra. The natural topology $\tau_p$ of $L^p_{\rm loc}({\mb R},m)$ is then defined as the inductive limit of the norm topologies of the spaces $L^p(K)$, when $K$ runs in the family of compact subsets of ${\mb R}$. 
		
		Let $\M$ denote the space of all ips-forms on $(L^p_{\rm loc}({\mb R},m), C_b({\mb R}))$ whose restriction to $L^p(K,m)$ is continuous for every compact subset $K\subset {\mb R}$.  Then, if $p\geq 2$, one can easily prove that $\M$ can be described by functions of $L^s_{\rm loc}({\mb R},m)$ where, as before, $s=\frac{p}{p-2}$ (again, $\frac10=\infty$).  It is easily seen that $\M$ is strongly well-behaved. In this case, the set of $\M$-bounded elements is the C*-algebra $L^\infty({\mb R},m)$.
		The pair $(L^p_{\rm loc}({\mb R},m), L^\infty({\mb R},m) )$ is a locally convex quasi GA*-algebra.
		
		\end{example}

			\medskip
			The following theorem motivates in our opinion the attention devoted to locally convex quasi GA*-algebras.
			\begin{theorem}
				Let	$(\A[\tau],\Ao)$ be a locally convex quasi GA*-algebras with unit and a well-behaved $\M\subset\IA$. Then:
				\begin{itemize}
					\item[(a)] every $\vp \in \M$ is jointly $\tau$-continuous;
					\item[(b)] every $\M$-regular *-representation of $(\A[\tau],\Ao)$ is $(\tau,{\sf t}_{s^\ast})$-continuous;
					\item[(c)]the set $\bdds$ of bounded elements is a C*-algebra  with respect to the norm $\|\cdot\|\rnb$.
				\end{itemize}
			\end{theorem}
			\begin{proof} 
				(a): Each $\vp$ is $\tau^\FF_*$-continuous by the construction itself of $\tau^\FF_*$; the statement then follows from the assumption $\tau \approx \tau^\FF_*$.\\
				(b): This follows from (a). Indeed if $\pi$ is $\M$-regular, then for every $\xi\in \D_\pi$, the sesquilinear form $\vp_\xi$ (see \eqref{eqn_vpxi}) is in $\M$; then, it is $\tau^\FF_*$-continuous. Then, there exists $\F\in\FF$ such that
					$$|\ip{\pi(a)\xi}{\pi(b)\xi}| \leq p_*^\F(a)p_*^\F(b)$$ hence  $$\| \pi(a)\xi\|\leq p_*^\F(a), \quad\mbox{and }\quad \| \pi(a^*)\xi\|\leq p_*^\F(a^*)=p_*^\F(a),\quad\forall a\in \A$$ then for every $\xi \in \D_\pi$ there exists $\F\in\FF$ such that  $$p^*_\xi (\pi(a))= \max\{\|\pi(a)\xi\|, \|\pi(a)\ad\xi\|\}\leq p_*^\F(a). $$ 
				(c): We have just to prove the completeness of the set $\bdds$  with respect to the norm $\|\cdot\|\rnb$. Let  $\{a_n\}\subset\bdds$ be a $\|\cdot\|\rnb$-Cauchy sequence, then for every $\epsilon>0$ there exists $n_\epsilon\in\mathbb{N}$ such that forall $n,m\geq n_\epsilon$ it is both $\|a_n-a_m\|\rnb<\epsilon$ and $\|a_n^*-a_m^*\|\rnb<\epsilon$.\\ Since $\{a_n\}\subset\bdds$, for every $\vp\in\M$ and every $x_0\in\Ao$, it is $$\vp((a_n-a_m)x,(a_n-a_m)x)\leq (\|a_n-a_m\|\rnb)^2\vp(x,x), \,\forall n,m\in \mathbb{N}$$ and $$\vp((a_n^*-a_m^*)x,(a_n^*-a_m^*)x)\leq (\|a_n^*-a_m^*\|\rnb)^2\vp(x,x), \,\forall n,m\in \mathbb{N}$$ hence, if $\F\in\FF$: $$\sup_{\vp\in\F}\vp((a_n-a_m)x,(a_n-a_m)x)^{1/2}\leq \|a_n-a_m\|\rnb\sup_{\vp\in\F}\vp(x,x)^{1/2},\,\,\forall n,m\in \mathbb{N}$$
				 and $$\sup_{\vp\in\F}\vp((a_n^*-a_m^*)x,(a_n^*-a_m^*)x)^{1/2}\leq \|a_n^*-a_m^*\|\rnb\sup_{\vp\in\F}\vp(x,x)^{1/2},\,\,\forall n,m\in \mathbb{N}$$
				  by the previous inequalities, for every $\F\in \mathfrak{F}$, we get $$p^\F_*(a_n-a_m) =\max\left\{p^\F(a_n-a_m), p^\F((a_n-a_m)^*)\right\}<\epsilon p^\F(\id),\, \forall n,m\geq n_\epsilon.$$  Then, $\{a_n\}$ is a $\tau^\FF_*$-Cauchy sequence. Since $\A$ is $\ltaum^\mathfrak{F}_*$-complete, there exists $a\in\A$ such that $a_n\stackrel{\tau^\FF_*}{\to} a$.
				 
				  The limit $a$ is $\M$-bounded; indeed, if $\vp \in\M$ and $x\in \Ao$, we have
				  $$\vp(ax,ax)= \lim_{n\to \infty} \vp(a_nx,a_n x)\leq \limsup_{n\to\infty}{(\|a_n\|\rnb)}^2 \vp(x,x).$$ The sequence $\{\|a_n\|\}$ is  Cauchy too and  bounded; therefore $a$ is $\M$-bounded. 
				  
				  It remains to prove that $\|a_n-a\|\rnb\to 0$ as $n\to \infty$. For every $\epsilon>0$ let $n, m>n_\epsilon$ then
				  $\|a_n-a_m\|\rnb <\epsilon$. Now let $m\to \infty$, then 
				  $$ \|a_n- a\|\rnb = \lim_{m\to \infty} \|a_n-a_m\|\rnb \leq \epsilon.$$
			\end{proof}

			{\section*{Conclusion}In this paper we have constructed some topologies on a quasi *-algebra $(\A,\Ao)$ starting from a sufficiently rich family of sesquilinear forms that behave regularly. 
				This study led us to introduce a new class of locally convex quasi  *-algebras, that we have named GA* since their definition closely recalls that one of A*-algebras. Several questions remain however still open. We mention some of them.
				
				(a) When does a (locally convex) quasi *-algebra $(\A,\Ao)$ possess a sufficient  family $\M$ of sesquilinear forms of $\IA$? 
				
				(b) Under what conditions is a locally convex quasi *-algebra $(\A[\tau],\Ao)$ a locally convex quasi GA*-algebra? We already know that there exist Banach quasi *-algebras $(\A[\|\cdot\|],\Ao)$ for which the set of continuous elements of $\IA$ reduces to $\{0\}$ \cite[Example 3.1.29]{FT_book} and the sesquilinear forms of a well-behaved family $\M$ of ips-forms are automatically continuous in a  locally convex quasi GA*-algebra. Hence, in general, the two notions do not coincide.
				
				(c) Under which conditions is it possible to lighten the definition of {\em well-behaved} 	family of ips-forms (Definition \ref{def:wbehaved}) by removing ({\sf wb}$_3$) and/or ({\sf wb}$_4$)?
				
				We hope to discuss these problems in a future paper.}
			
			\bigskip
			{\bf{Acknowledgements:} } This work has been done within the activities of Gruppo UMI Teoria dell’Approssimazione e Applicazioni and of GNAMPA of the INdAM.

\end{document}